\documentclass[11 pt, reqno]{amsart}
\usepackage{amsmath}
\usepackage[alphabetic]{amsrefs}
\usepackage{amsthm}
\usepackage[english]{babel}
\usepackage{amssymb}
\usepackage{graphics}
\usepackage{epsfig}
\usepackage{amscd}
\usepackage{amsfonts}
\usepackage{amsbsy}
\usepackage{float}
\usepackage{caption}
\usepackage{subcaption}
\usepackage{overpic}
\usepackage{dsfont}
\usepackage{MnSymbol,wasysym}

\usepackage{xcolor}

\usepackage[colorlinks=true,linkcolor=blue]{hyperref}

\definecolor{asparagus}{rgb}{0.53,0.66, 0.42}
\definecolor{xanadu}{rgb}{0.45, 0.53, 0.47}
\definecolor{alizarin}{rgb}{0.82, 0.1, 0.26}
\definecolor{ao}{rgb}{0.0, 0.5, 0.0}
\definecolor{azure}{rgb}{0.0, 0.5, 1.0}
\definecolor{awesome}{rgb}{1.0, 0.13, 0.32}
\definecolor{greenish}{rgb}{0.1, 0.5, 0.1}
\definecolor{armygreen}{rgb}{0.29, 0.33, 0.13}

\newcommand{\bluecom}[1]{{\color{blue} #1}}

\newcommand{\R}{\mathds{R}}

\newtheorem {theorem} {Theorem}
\newtheorem {prop} [theorem] {Proposition}

\newtheorem {lemma} [theorem] {Lemma}

\newtheorem {remark} {Remark}

\usepackage{color}

\begin{document}

\title[hip-hop  solutions]{Periodic oscillations in the restricted hip-hop $2N+1$-body problem}

\author[Andr\'es Rivera,  Oscar Perdomo, Nelson Castañeda]
{Andr\'es Rivera $^1$ , Oscar Perdomo $^2$, Nelson Castañeda $^2$.}

\address{$^1$ Departamento de Ciencias Naturales y Matem\'aticas.
Pontificia Universidad Javeriana, Facultad de Ingenier\'ia y Ciencias,
Calle 18 No. 118--250 Cali, Colombia.}
\email{ amrivera@javerianacali.edu.co, perdomoosm@cssu.edu, CastanedaN@ccsu.edu
}\emph{}
\address{$^2$ Departament of Mathematical Scinces. Central Connecticut State University. New Britain. CT 06050. Connecticut, USA.}



\date{}
\dedicatory{}
\maketitle

\begin{center}\rule{0.9\textwidth}{0.1mm}
\end{center}
\begin{abstract}
We prove the existence of periodic solutions of the restricted $(2N+1)$-body problem when the $2N$-primaries move on a periodic Hip-Hop solution and the massless body moves on the line that contains the center of mass and is perpendicular to the base of the antiprism formed by the $2N$-primaries. 
\end{abstract}
\begin{center}\rule{0.9\textwidth}{0.1mm}
\end{center}

\section{Introduction.}

In a classical restricted $(n+1)$-body problem the motion of one body (that we call the \textit{plus one body}) with mass $m_{0}\approx 0$ is affected only by the gravitational force of the other $n$-bodies and does not perturb the motion of the $n$-bodies, which interact only with gravitational forces. When the $n$-bodies have the same mass $m>0$ they are usually called \textit{primaries}. This mathematical model can be used to describe motions of comets, spacecraft and asteroids (see \cite{SB} and the references therein). In the literature, there is a remarkable example of a restricted $3$-body problem called the \textit{Sitnikov problem} (proposed in 1960 by K.A. Sitnikov \cite{Sit}). Here the two primaries move in elliptic orbits  lying in the $xy$-plane around the center of mass (barycenter) as solutions of the 2-body problem. Finally, the plus one body moves along the $z$-axis passing through the barycenter of the primaries. The Sitnikov problem deals with the orbits of the plus one body. Since its formulation, this particular 3-body problem has been treated in several papers \cite{A,BLO,Corbera,Corbera-Llibre,JLEB,LlOR,OR,Rob,Sit} from both, numerical and analytical point of view. Concerning the existence of periodic motions for the plus one body in the Sitnikov problem, we refer to \cite{LlOR,OR} where the authors prove the existence of families of symmetric periodic orbits which depend continuously on the eccentricity of the primaries, by means of the Leray-Schauder continuation method. Using the same technique, in \cite{R} the author proves that these families also exist in a particular generalization of the Sitnikov problem where the number of primaries is $n\geq 2$. More precisely, in \cite{R} the authors find the existence of periodic motions in a restricted $(n+1)$-body problem where each primary body is at the vertex of a regular $n$-gon and moves on elliptic orbits lying in the $xy$-plane around their barycenter and the plus one body moves on the $z$-axis. The link \url{https://youtu.be/RjlZpqDFsDM} leads to a video showing some of these periodic motions when the plus one body remains still at the origin (the center of mass). Moreover, when the eccentricity of orbits is zero, all the primaries have the same circular orbit, i.e., the primaries perform a choreography.
	
In this paper we are considering also $n$-primaries but this time they do not move on a plane, they are periodic Hip-hop solutions of the $n$-body problem. The existence of these solutions were studied in \cite{BC,BCPS} and more recently, in \cite{PRAC} we found the existence of families of periodic Hip-hop solutions for a $2N$-body problem. A hip-hop solution satisfies $i)$ All the bodies have the same mass. $ii)$ at every instante of time $t$, $N$ of the bodies are at the vertices of a regular $N$-gon contained in a plane $\Pi_{1}(t)$  and the other $N$-bodies are at the vertices of a second regular $N$-gon contained in a plane $\Pi_{2}(t)$ which is obtained from the first $N$-gon by the reflection in a fixed plane $\Pi_{0}$ followed by a rotation of $\pi/N$ radians around a fixed-line $l_{0}$ perpendicular to the three planes and passing through the center of the two $N$-gons. When $\Pi_{1}(t)\neq \Pi_{2}(t)$, the $2N$-bodies are at the vertices of an antiprism that degenerates to a regular $2N$-gon when $\Pi_{1}(t)= \Pi_{2}(t)$. At any time $t$, the oriented distance from the plane $\Pi_{0}$ to $\Pi_{1}(t)$ is given by a function $d(t)$ and the distance of any primary body to the line $l_{0}$ is given by a function $r(t)$. Without loss of generality, it is possible to assume that the line $l_{0}$ is the $z$-axis and the plane $\Pi_{0}$ is the $xy$-plane.  Let us explain in more detail the new type of restricted $(2N+1)$-body problem which can be thought of as a sort of Sitnikov problem.
	
\textbf{The restricted hip-hop body problem}. Consider $2N$-bodies with equal mass (called primaries) located at any time at the vertices of a regular antiprism, moving as a periodic Hip-hop solution of a $2N$-body problem and a massless body moving on the straight line orthogonal to the two regular $N$-gons that form the antiprism and passes through their center of mass. \textit{The restricted hip-hop body problem} will consist in describing the motion of the massless body.

\begin{figure}[h]
	\centerline
	{\includegraphics[scale=0.22]{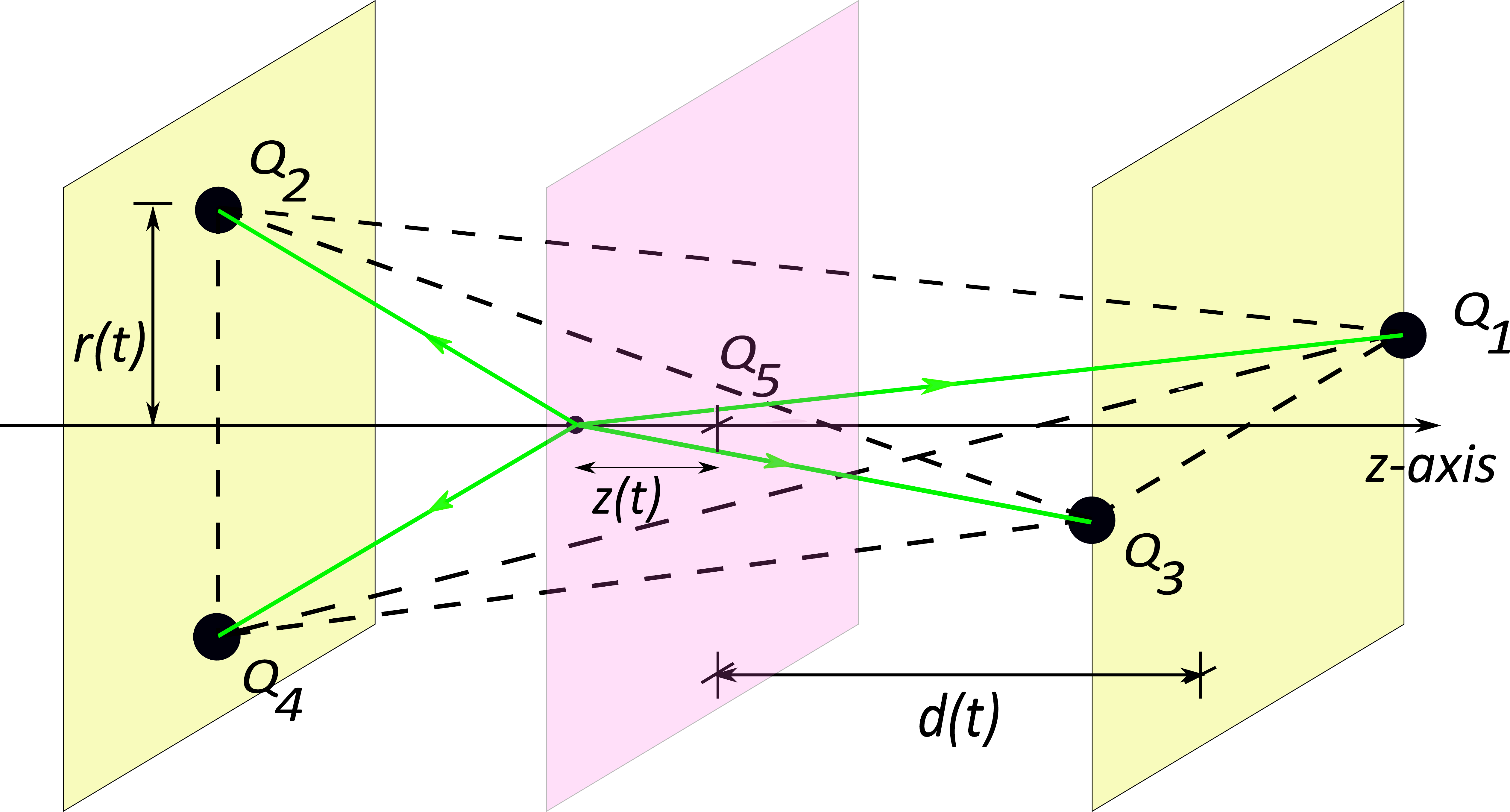}}
	\caption{Restricted hip-hop body problem with four primaries. At any time, the primaries are placed at the vertices of a regular $4$-gonal antiprism and the plus one body moves on the line orthogonal to the two $2$-gons that are the base face of the antiprism.}
	\label{fex1}
\end{figure}

The main objective of this paper is to analytically prove the existence of symmetric periodic orbits for the restricted antiprism body problem. 
To this end, the paper is divided as follows. In Section \ref{section 2} we deduce the equation of motion for the $2N$-primaries and the plus one body. In Section 3, we apply the same techniques found in \cite{PRAC} by reducing the problem of the existence of periodic solutions to the problem of solving a system of three equations in four variables. The main results of the paper are contained in Section \ref{section 4} and some numerical periodic solutions of the restricted hip-hop problem for $6$ primaries are given in Section \ref{section 5}.


\section{The differential equations}\label{section 2}

We will follow the same set up as in the references \cite{PRAC} and \cite{BC}. To describe the motion of the primaries we consider $ Q_ {1}, Q_ {2}, \ldots, Q_ {2 N} $ bodies with equal mass $ m> 0 $, located on the vertices of a regular anti-prism. If $\mathbf{r}_{j}(t)$ is the position of the body $Q_{j}, j=1,\ldots, 2N$ at each instant $t$  then
\[
\mathbf{r}_{j}(t)=\mathcal{R}^{j-1}\mathbf{r}_{1}(t), \quad j=1,\dots, 2N,
\]
where $r_1(t)=(r(t) \cos(\theta(t)),r(t) \sin(\theta(t)), d(t))$ and $\mathcal{R}$ is a rotation/reflection matrix given by
\begin{equation*}
\label{matriz de A-}
\mathcal{R}=
\begin{pmatrix}
\cos (\frac{\pi}{N})& -\sin(\frac{\pi}{N}) & 0 \\
\sin(\frac{\pi}{N}) & \cos (\frac{\pi}{N}) & 0 \\
0& 0& -1
\end{pmatrix}.
\end{equation*}

We have that the functions $r(t)$, $\theta(t)$ and $d(t)$ that define $r_1(t)$ satisfy the equations

\begin{equation}\label{cilindricas}
\begin{cases}
\begin{split}
\ddot{r} &= f(r, d),\\
\ddot{d} &=g(r, d), \\
\dot{\theta} &=a/r^2,
\end{split}
\end{cases}
\end{equation}

where $\displaystyle{a}$ is the angular momentum of the system and
\begin{equation*}
\begin{split}
f(r, d) &=\frac{a^2}{r^3} - 2rm \sum_{k=1}^{2N-1} \frac{\sin ^2(k \pi/ 2N)}{\left[4r^2 \sin^2(k \pi / 2N) + ((-1)^k - 1)^2 d^2\right]^{3/2}}, \nonumber \\
g(r, d) &= -\frac{m\,d}{2} \sum_{k=1}^{2N-1} \frac{((-1)^k - 1)^2}{\left[4r^2 \sin^2(k \pi / 2N) +((-1)^k - 1)^2d^2\right]^{3/2}}.
\end{split}
\end{equation*}

Now that we have a description of the $2 N$ primaries, we describe the motion of the plus one body. Since this body moves on the $z$-axis, then its motion is describe\bluecom{d} by the map $r_{2N+1}(t)=(0,0,z(t))$. Notice that the equation of the primaries do not change because of our assumption, also we notice in this model the mass of the plus one body is irrelevant because it cancels out from the equations.

   A direct computation shows that $z(t)$ satisfies the following differential equation. 
\begin{equation}\label{eqplusone}
\ddot{z}=h(r,d,z)=-m N  \left[ \frac{z - d}{((z - d)^2 + r^2)^{3/2}}+ \frac{z + d}{((z + d)^2 + r^2)^{3/2}}\right]
\end{equation}

Observe that our whole system is partially decouple\bluecom{d} and we can solve first for $r(t),d(t),z(t)$ using the equations

\begin{equation}\label{maineq}
	\begin{cases}
		\begin{split}
			\ddot{r} &=  f(r, d), \qquad r(0)=r_0,\quad \dot{r}(0)=0,\\
			\ddot{d} &= g(r, d), \qquad d(0)=0,\quad \dot{d}(0)=b,\\
			\ddot{z} &=h(r,d,z), \quad z(0)=0,\quad \dot{z}(0)=u.
		\end{split}
	\end{cases}
\end{equation}

and then find $\theta(t)$ as
\[
\theta(t)=\int_{0}^{t}\frac{a}{r^{2}(s)}ds.
\]

From now on we assume $r_0, m$ and $N$ fixed and we denote by 
\begin{eqnarray}\label{DefRDZ}
R(a,b,t)=r(t), \quad D(a,b,t)=d(t), \quad \Theta(a,b,t)=\theta(t) \quad \text{and} \quad Z(a,b,u,t)=z(t),
\end{eqnarray}
the solutions of the system (\ref{maineq}) with initial conditions
\begin{equation}\label{initial conditions}
	r(0)=r_0,\quad \dot{r}(0)=0,\quad d(0)=0,\quad \dot{d}(0)=b,\quad \theta(0)=0 ,\quad z(0)=0, \quad \dot{z}(0)=u.
\end{equation}



\section{Preliminary results}\label{section3}

We will need the following lemma. For a detailed proof see \cite{P}

\begin{lemma}\label{limint} Let $f:(t_0-\epsilon,t_0 \bluecom{+}\epsilon)\to \mathbb{R}$ be a smooth function that satisfies $f(t_0)=f^\prime(t_0)=0$ and $f^{\prime\prime}(t_0)=-2a<0$. For any small $c>0$ we denote by $t_1(c), t_2(c)$  the two zeroes of $f(t)+c$ such that  $t_1(c)<t_0<t_2(c)$ and  $f(t) + c > 0 $ for all $ t \in (t_1(c), t_2(c)).$ Then

$$\lim_{c\to 0^+}\int_{t_1(c)}^{t_2(c))} \frac{dt}{\sqrt{f(t)+c}} =\frac{\pi}{\sqrt{a}}.$$
\end{lemma}

Next we show how some symmetries of the differential equations can be used to detect periodicity. The following symmetries can be seen directly from the definitions of the functions involved
\[
f(r,d)=f(r,-d), \quad g(r,d)=-g(r,-d), \quad h(r,d,z)=-h(r,-d,z) \quad \text{and} \quad h(r,d,z)=-h(r,d,-z) 
\]

\begin{lemma}\label{symmetries}
Let us consider $R,D$ and $Z$ the functions defined in Equation \eqref{DefRDZ} and  let us denote by $\dot{R}(a,b,T)=\frac{\partial R}{\partial T}(a,b,T)$.
If for some $(a,b,u,T)$
	\begin{equation}\label{ic}
	\begin{split}
		\dot{R}(a,b,T)=0,\quad D(a,b,T)=0,\quad Z(a,b,u,T)=0,
	\end{split}
	\end{equation}
	 then, $r(t)=R(t,a,b)$, $d(t)=D(a,b,t)$ and $z(t)=Z(a,b,u,t)$ is a  $2T$-periodic solution of the system of  differential equations \eqref{maineq}.
\end{lemma}
\begin{proof}

	A direct computation proves that the functions $\hat{R},\hat{D}$ and $\hat{Z}$ given by
	\[
	\hat{R}(a,b,t)=R(a,b,-t), \quad 	\hat{D}(a,b,t)=-D(a,b,-t), \quad 
	\hat{Z}(a,b,u,t)=-Z(a,b,u,-t),
	\]
satisfy (\ref{maineq}). By uniqueness it follows that
	\[
	R(a,b,t)=R(a,b,-t) \quad 	D(a,b,t)=-D(a,b,-t), \quad 
	Z(a,b,u,t)=-Z(a,b,u,-t),
	\]
A similar argument using the assumption (\ref{ic}) will give us symmetries around $t=T$ that allow to show that $R,D$ and $Z$ are $2T$-periodic.
\end{proof}

The following Theorem was proven in \cite{PRAC} and shows the periodicity of some hip-hop solutions. A similar result can be found in \cite{BCPS}.
Before, stating the next theorem let us define
\begin{equation}\label{sumas}
	\alpha_{N}=\frac{1}{16}\sum_{k=1}^{2N-1} \frac{((-1)^{k}-1)^{2}}{\sin^{3}(\frac{k\pi}{2N})}, \quad \gamma_{N}= \frac{1}{4}\sum_{k=1}^{2N-1} \frac{1}{\sin(\frac{k\pi}{2N})}.
\end{equation}

\begin{theorem}\label{mainprimaries} Let $N>1$ and $m, r_0>0$ fixed. There exists $\hat{b}>0$ and a pair of functions $T_{1},a:(-\hat{b},\hat{b})\to \mathbb{R}$, with
\begin{eqnarray}\label{defa*T*}
 T_{1}(0)=T_1^\star=\pi \sqrt{\frac{r_{0}^3}{m\alpha_{N}}}, \quad \text{and} \quad a(0)=a^\star=\sqrt{m\gamma_{N}r_0},
\end{eqnarray}
such that $\displaystyle{R(a(b), b, t)=r(t)}$   and $D(a(b), b, t)=d(t)$  are  $2T_1(b)$-periodic solutions of the first two equations in system \eqref{maineq}. Moreover, $\dot{R}(a(b),b,T_1(b))=D(a(b),b,T_1(b))=0$ and therefore, both functions $r(t)$ and $d(t)$ are even with respect to $t=T_1(b)$.    


\end{theorem}

\section{Main results}\label{section 4}

This section shows our main result on the periodicity of the $(2N+1)$-restricted problem that we are considering. Our main tool  are the Implicit Function Theorem and a compactness argument. We start this section by proving a 
proposition and a lemma   that are needed in our main compactness argument.

The following Proposition shows the periodicity of some solution of the restricted $(2N+1)$-body problem in the case that the primaries $2N$ move in a circle. For the case $N=1$, $r_0=1/2$ and $m=1$, explicit solutions in terms of elliptic functions are found in \cite{BLO}.

\begin{prop}\label{periodmassless} Let $a^*$ be the constant defined in Equation \eqref{defa*T*}. For any real number $u$ with $|u|<\sqrt{\frac{4 m N}{r_0}}$, the function $z(t)=Z(a^*,0,u,t)$ is periodic with period $2T(u)$,  with 

$$T(u)=\int_{-t_{1}(c)}^{t_{1}(c)}\frac{1}{\sqrt{4mN(z^2+r^2_{0})^{-1/2}+c}}\, dz,$$

where  $c=u^2-\frac{4mN}{r_0}$ and $t_1(c)=\sqrt{\frac{16m^2N^2}{c^2}-r_0^2}$. Moreover 

$$\lim_{u\to 0}T(u)=\pi\sqrt{\frac{r_0^3}{2mN}}\quad\hbox{and}\quad \lim_{u\to \pm  \sqrt{\frac{4 m N}{r_0}}}T(u)=\infty.$$
\end{prop}

\begin{proof}
When $b=0$ and $a=a^*$ we have that $d(t)\equiv 0$ and $r(t)\equiv r_0$, therefore, using the last equation in (\ref{maineq}) we have that $z(t)=Z(a^*,0,u,t)$ satisfies the equation 

\begin{eqnarray}\label{ez1}
\ddot{z}= -  \frac{2m N z }{(z^2 + r_0^2)^{3/2}}.
\end{eqnarray}

Multiplying Equation (\ref{ez1}) by $\dot{z}$ we obtain that $z(t)$ satisfies the equation 

\begin{eqnarray}\label{ez2}
\dot{z}^2=   \frac{4m N }{(z^2 + r_0^2)^{1/2}}+c=f(z)+c,\quad\hbox{where $f(z):= \frac{4m N }{(z^2 + r_0^2)^{1/2}}$}
\end{eqnarray}

It is clear that if $ c > 0 $ then the function $ z(t) $ has no critical points and cannot be periodic. Therefore, we consider only the case when $ c < 0.$ 

Our initial conditions at $t=0$ on $z(t)$ implies that $u^2=\frac{4mN}{r_0}+c$. Notice that $0<f(z)\le \frac{4mN}{r_0}$. Then, when $|u|<\sqrt{ \frac{4mN}{r_0}}$, $u^2=\frac{4mN}{r_0}+c< \frac{4mN}{r_0}$. Therefore, for negative values of $c$ greater than $- \frac{4mN}{r_0}$, the zeroes of $f(z)+c$ are given by $\pm t_1(c)$ with

\[
t_{1}(c)=\sqrt{\frac{16m^2N^2}{c^2}-r_0^2},
\] 
and the period function is $2T(u)$, with $T(u)$ given by
\[
T(u)=\int_{-t_{1}(c)}^{t_{1}(c)}\frac{1}{\sqrt{f(z)+c}}\, dz\bluecom{.}
\]
Now we compute $\displaystyle{\lim_{u\to 0}T(u)}$. To this end, we apply Lemma \ref{limint}
\[
\begin{split}
\lim_{u\to 0}T(u)=&\lim_{c\to -\frac{4mN}{r_0}}\int_{-t_{1}(c)}^{t_{1}(c)}\frac{1}{\sqrt{f(z)+c}}dz,\\ &=\frac{\pi}{\sqrt{-f^{\prime \prime}(0)/2}}= \pi\sqrt{\frac{r_0^3}{2mN}}.
\end{split}
\]
On the other hand, since
\[
\lim_{M\to \infty}\int_{-M}^{M}\frac{1}{\sqrt{f(z)}}dz=\lim_{M\to \infty}\frac{1}{\sqrt{m N}}\int_{0}^{M}(z^2 + r_0^2)^{1/4}dz=\infty,
\]

then, for $K$ any positive number there exists $M=M(K)$ such that

 \[
\int_{-M}^{M}\frac{1}{\sqrt{f(z)}}dz>K.
	\] 

Now, choose $-\frac{4mN}{r_0}<c<0$ close to zero ($|u|$ close to  $\sqrt{\frac{4 m N}{r_0}}$  ) such that $|t_{1}(c)|>M(K)$. Therefore, 

\[
	\int_{-t_{1}(c)}^{t_{1}(c)}\frac{1}{\sqrt{f(z)+c}}dz \,> \, \int_{-M}^{M}\frac{1}{\sqrt{f(z)+c}}dz \, > \, \int_{-M}^{M}\frac{1}{\sqrt{f(z)}}dz> K,
	\]

proving that $\displaystyle{\lim_{u\to \pm  \sqrt{\frac{4 m N}{r_0}}}T(u)=\infty.}$
\end{proof}

%
%
Since the equation of the motion of the primaries does not involve the function $z=z(t)$, we can think of the functions $R(a,b,T)$ and $D(a,b,T)$ as functions on the space  $\{(a,b,u,T): a,b,u,T\in \mathbb{R}\}$ by setting $R(a,b,u,T)=R(a,b,T)$ and $D(a,b,u,T)=D(a,b,T)$. We can also think of the function $T(u)$ in Proposition \ref{periodmassless} as a function defined on the segment $L=\{(a*,0,u): -\sqrt{\frac{4 m N}{r_0}}<u<\sqrt{\frac{4 m N}{r_0}}\}$  that satisfy $Z(a^*,0,u,T(a^*,0,u))=0$. 

The next lemma shows that the function $T$ 
can be extended to an open set in the 3 dimensional space $\{(a,b,u):a,b,u\in\mathbb{R}\}$ containing the segment  $L$ to a function in three variables $T_2$ that satisfies  $Z(a,b,u,T_2(a,b,u))=0$.

%
%

\begin{lemma}\label{function T2} For any $0<u_1<u_2<\sqrt{\frac{4mN}{r_0}}$, there exist positive numbers $\epsilon_1$, $\epsilon_2$ and a smooth function 
$$T_2:(a^{*}-\epsilon_1,a^{*} \bluecom{+} \epsilon_1)\times (-\epsilon_2,\epsilon_2)\times [u_1,u_2]\longrightarrow \mathbb{R},$$

such that $Z(a,b,u,T_2(a,b,u))=0.$
\end{lemma}

\begin{proof}
	For any $u_{0}$ in $[u_1,u_2]$ we have that
	\[
	Z(a^{*},0,u_{0},T(u_{0}))=0, \quad \text{and} \quad Z_{t}(a^{*},0,u_{0},T(u_{0}))\neq 0, 
	\]
	By the Implicit Function Theorem there exist positive numbers $\epsilon_{j}(u_{0}), j=1,2,3$, and  a smooth function $T_{2,u_{0}}:W_{u_0}\to \R $ with
	\[
	W_{u_0}=(a^{*}-\epsilon_{1}(u_{0}),a^{*}+\epsilon_{1}(u_{0}))\times (-\epsilon_{2}(u_{0}),\epsilon_{2}(u_{0}))\times (u_{0}-\epsilon_{3}(u_{0}), u_{0}+\epsilon_{3}(u_{0})),
	\] 
	such that the only solution of $Z(a,b,u,T)=0$ with $(a,b,u,T)\in W_{u_0}$ is given by $(a,b,u,T_{2,u_0}(a,b,u))$.
	
We have that  $\displaystyle{\left\{(u_{0}-\epsilon_{3}(u_{0}), u_{0}+\epsilon_{3}(u_{0}))\right\}_{u_{0}\in[u_1,u_2]}}$ is an open cover of  $[u_1,u_2]$. By compactness there exist $u^{i}_{0},i=1,\dots, k$ such that $\displaystyle{\left\{(u^{i}_{0}-\epsilon^{i}_{3}(u_{0}), u^{i}_{0}+\epsilon^{i}_{3}(u_{0}))\right\}_{i=1,\dots,k}}$ covers $[u_1,u_2]$. Take $\epsilon_1=\min\left\{\epsilon^{i}_{1}(u_{0}), i=1,\dots, k\right\}$ and $\epsilon_2=\min\left\{\epsilon^{i}_{2}(u_{0}), i=1,\dots, k\right\}$ and define $T_{2}:W\to \R,$ with 
\[
 W=(a^{*}-\epsilon_1,a^{*}+\epsilon_1)\times (-\epsilon_2,\epsilon_2)\times [u_1,u_2]\quad\hbox{and}\quad
T_{2}(a,b,u)= T_{2,u^{j}_{0}}(a,b,u)\bluecom{,}
\] 
where $j$ satisfies that $u\in I_j:= (u^{j}_{0}-\epsilon^{j}_{3}(u_{0}), u^{j}_{0}+\epsilon^{j}_{3}(u_{0})).$  Notice that $T_{2}$ is independent of $j$, because if $u\in I_{i}\cap I_{j}$  then $(a,b,u)\in W_{u^{i}_0}\cap W_{u^{j}_0}$
and
\[
Z(a,b,u,T_{2,u^{i}_{0}}(a,b,u))=Z(a,b,u,T_{2,u^{j}_{0}}(a,b,u))=0,
\]
which implies that $T_{2,u^{i}_{0}}(a,b,u)=T_{2,u^{j}_{0}}(a,b,u)$ by the uniqueness part fo the Implicit Function Theorem.

\end{proof}
\begin{remark}\label{rem1} Notice that for $a=a^{*}$ and $b=0$ we deduce that $T(u)=T_{2}(a^{*},0,u)$ for any $0<u<\sqrt{\frac{4mN}{r_0}}$.
\end{remark}
\begin{theorem}\label{main} For any $m$, $N$ and $r_0$ there exists a family of non-trivial  (with $b\ne 0$) periodic solution of  \eqref{maineq}.
\end{theorem}

\begin{proof} For each $m,N$ y $r_{0}$, let $k$  the first positive integer such that
	\[
	0<T^{*}_{2}<kT^{*}_{1},
	\]	
with
\[
T^{*}_{1}=\pi\sqrt{\frac{r_0^3}{m \alpha_N}} \quad \text{\and} \quad T^{*}_{2}=\pi\sqrt{\frac{r_0^3}{2m N}} \bluecom{.} 
\]	
From Proposition \ref{periodmassless} and the continuity of $T(u)$ there exist $0<u_1<u_2<\sqrt{\frac{4mN}{r_0}}$ such that 
\[
T(u_{1})<kT^{*}_{1}<T(u_{2}).
\]
Define
\[
\tau_{1}:=(kT^{*}_{1}+T(u_{1}))/2, \quad \tau_{2}:=(kT^{*}_{1}+T(u_{2}))/2,
\]
and by Lemma \ref{function T2} the functions
\[
\begin{split}
h_{1}:\Omega\to \R, \quad h_{1}(a,b)=T_{2}(a,b,u_{1}),\\
h_{2}:\Omega\to \R, \quad h_{2}(a,b)=T_{2}(a,b,u_{2}),
\end{split}
\]
with $\Omega=(a^{*}-\epsilon_1,a^{*} \bluecom{+} \epsilon_1)\times (-\epsilon_2,\epsilon_2).$

From  Remark \ref{rem1} and continuity of $T_2$ there exist two neighboorhoods $\mathcal{U}_{i}\subset \Omega, i=1,2$ such that
\[
\begin{split}
h_{1}(a,b)<\tau_1 \quad \text{for} \quad (a,b)\in \mathcal{U}_{1},\\
h_{2}(a,b)>\tau_2 \quad \text{for} \quad (a,b)\in \mathcal{U}_{2}.\\
\end{split}
\]
Let $(a,b)\in \mathcal{U}=\mathcal{U}_{1}\cap \mathcal{U}_2$ and consider the open set $\mathcal{U}\times (\tau_{1},\tau_{2})$. By continuity of $T_1(b)$, there exists $0<\hat{\epsilon}<\epsilon_2$ such that $(a(b),b,T\bluecom{_1}(b))\in \mathcal{U}\times (\tau_{1},\tau_{2})$ for all $b\in (-\hat{\epsilon},\hat{\epsilon})$. Now for each fixed $b\in (-\hat{\epsilon},\hat{\epsilon})$ define the continuous function $g:[u_1,u_2]\to \R, \,u \to g(u)=T_{2}(a(b),b,u).$  Then
\[
\begin{split}
g(u_{1})=T_{2}(a(b),b,u_1)&=h_{1}(a(b),b)<\tau_1, \\
g(u_{2})=T_{2}(a(b),b,u_2)&=h_{2}(a(b),b)>\tau_2.
\end{split}
\]
Since $\tau_1<kT^{*}_{1}<\tau_2$, there exists $u_{*}\in [u_{1},u_{2}]$ such that $kT^{*}_{1}=T_{2}(a(b),b,u_{*})$. Moreover, for $b\in (-\epsilon,\epsilon)$ with $\epsilon=\min\left\{\hat{b},\hat{\epsilon}_{2}\right\}$, $\tau_1<kT_{1}(b)<\tau_2$ and there exists $u(b)\in [u_{1},u_{2}]$ such that $kT_{1}(b)=T_{2}(a(b),b,u(b)).$ Therefore,
\[
\dot{R}(a(b),b,kT_{1}(b))=0,\quad D(a(b),b,kT_{1}(b))=0,\quad Z(a(b),b,u(b),kT_{1}(b))=0,
\]
with $b\in (-\epsilon,\epsilon)$, implying by Lemma \ref{symmetries} the existence of a family of $2kT_{1}(b)$-periodic solutions of the system (\ref{maineq})

\end{proof}

\section{Some numerical solutions}\label{section 5}

In this section we show two numerical examples of periodic solutions of the restricted hip-hop problem. Using analytic continuation, it is possible to find a family of points $\{(a_i,b_i,T_i)\}$ emanating from the point $(a^*,0,T_1^*)$ such that $\dot{R}(a_i,b_i,T_i)=D(a_i,b_i,T_i)=0$, see \cite{PRAC}. From the symmetry of the hip-hop equations we have that for each one of these points, the functions $d(t)=D(a_i,b_i,t)$ and $r(t)=R(a_i,b_i,t)$ define a hip-hop solution with period $2T_i$. Our main result, Theorem \ref{main}, guarantees the existence of a family of points of the form  $\{(a_i,b_i,u_i,T_i)\}$ with the property that $\dot{R}(a_i,b_i,T_i)=D(a_i,b_i,T_i)=Z(a_i,b_i,u_i,T_i)=0$. The previous equations imply that not only $d(t)$ and $r(t)$ are periodic but also $z(t)=Z(a_i,b_i,u_u,t)$ is also a function with period $2T_i$ as well.  We have done analytic continuation to catch some interesting solutions. With the intension of borrowing some hip-hop solutions from paper \cite{PRAC}, we will take $N=3$, $r_0=2$ and $m=1$.
 
\subsection{Example 1.}  We can check that the choice $a_1=0.581722,b_1=0.81081,u_1=1.96752,T_1=6.53474$ is a numerical solution of the system
\begin{eqnarray}\label{system}
 \dot{R}(a,b,T)=D(a,b,T)=Z(a,b,u,T)=0.
\end{eqnarray}
For these values of $a$ and $b$, the primaries share a single trajectory, this is, the primaries solution is a choreography. As expected, the functions $d(t)=D(a_1,b_1,u_1,t)$, $r(t)=R(a_1,b_1,u_1,t)$ and $Z(t)=Z(a_1,b_1,u_1,t)$ are functions with period $2T_1$. Figure \ref{fex1} shows the functions $d(t)$, $r(t)$ and $z(t)$. Figure \ref{bex1} shows the (6+1)-bodies for different values of $t$.
\begin{figure}[h]
	\centerline
	{\includegraphics[scale=0.87]{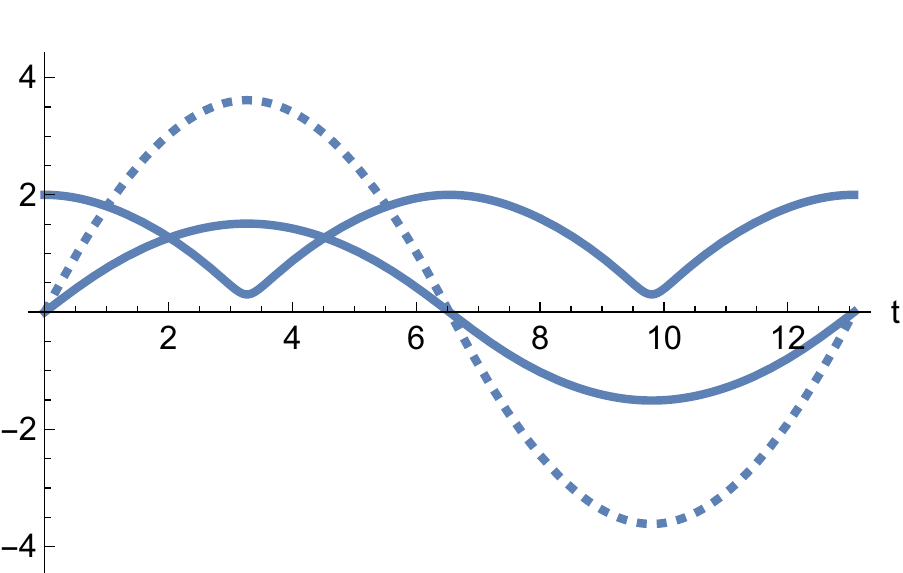}}
	\caption{Graph of the $2T_1$-periodic functions $d(t)$, $r(t)$ and $z(t)$ given by the parameters $a_1=0.581722$, $b_1=0.81081$, $u_1=1.96752$ and $T_1=6.53474$. The dashed graph is the function $z(t)$ and the function that starts at $r_0=2$ is the function $r(t)$.  }
	\label{fex1}
\end{figure}

\begin{figure}[h]
	\centerline
	{\includegraphics[scale=0.45]{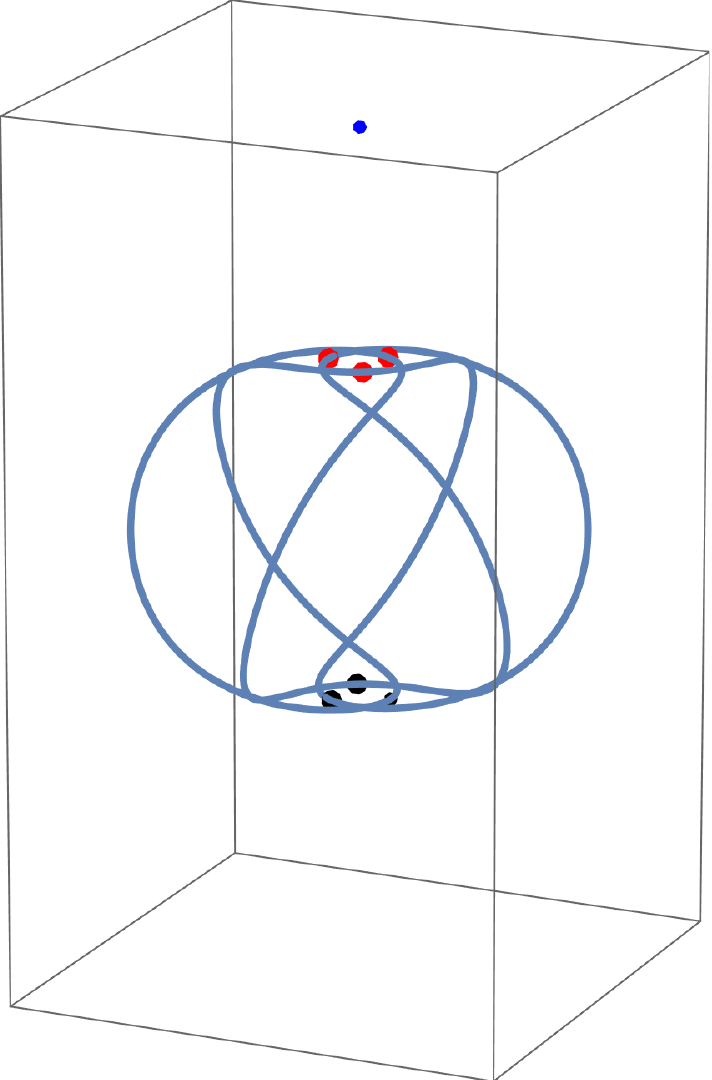} \includegraphics[scale=0.45]{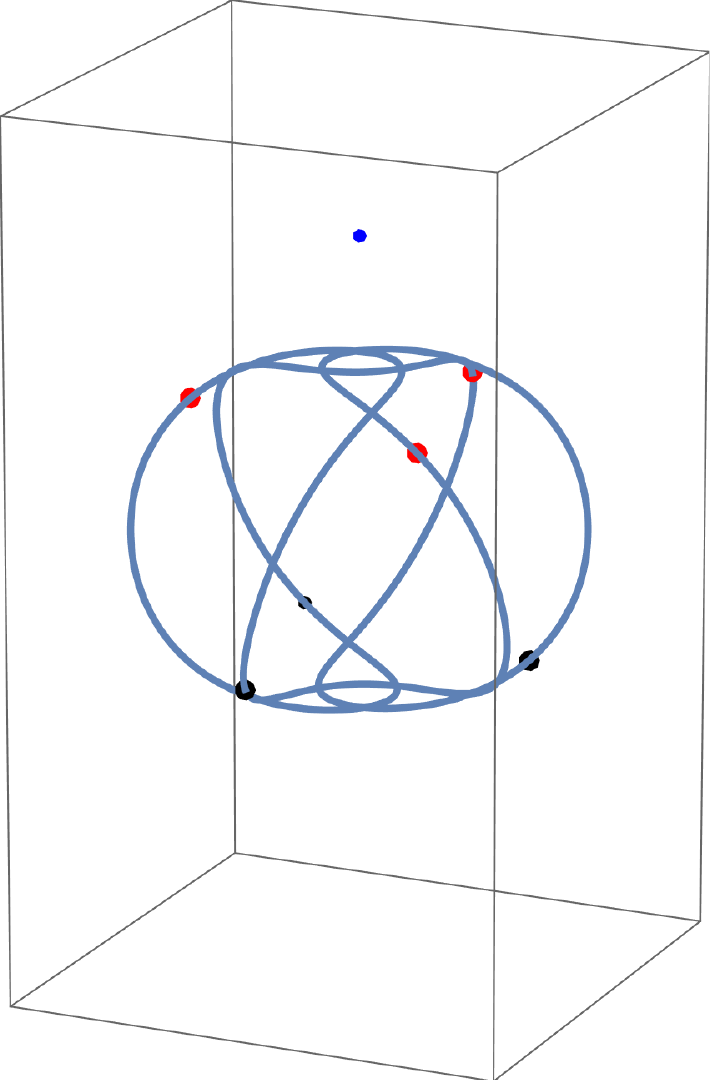}\includegraphics[scale=0.35]{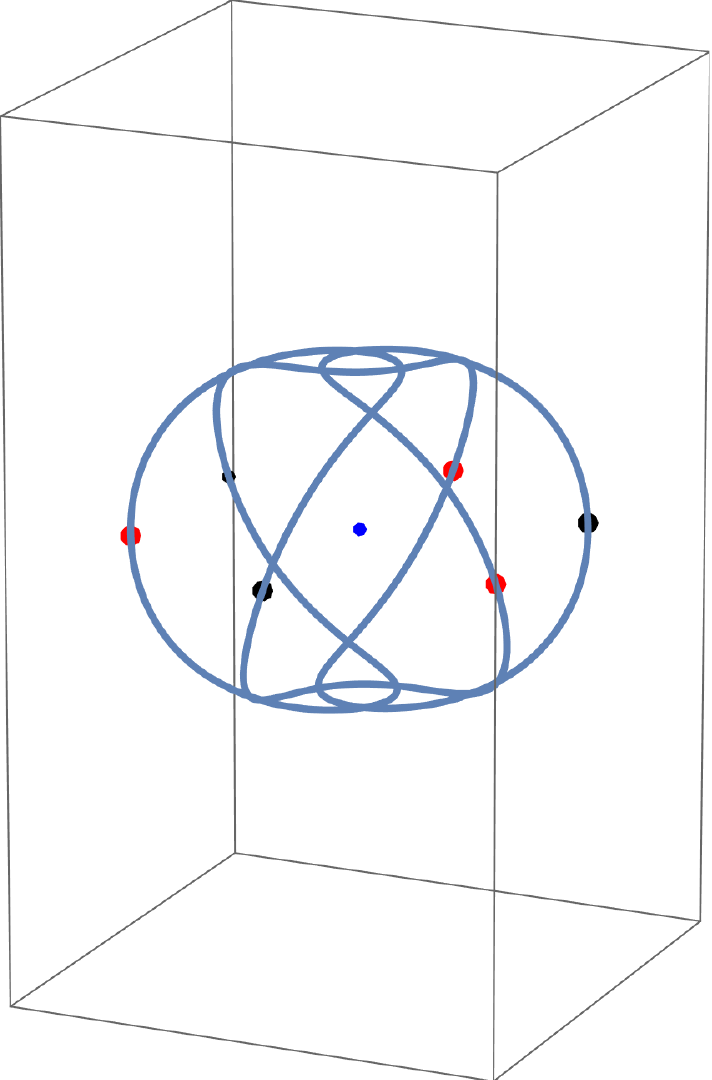}\includegraphics[scale=0.45]{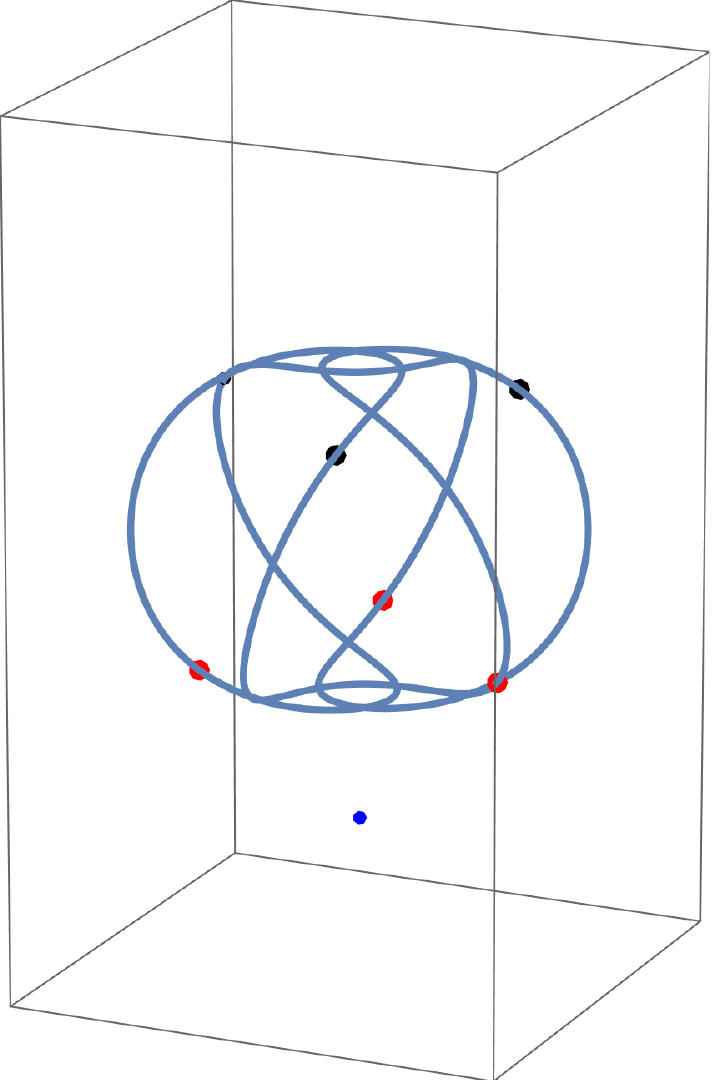}\includegraphics[scale=0.35]{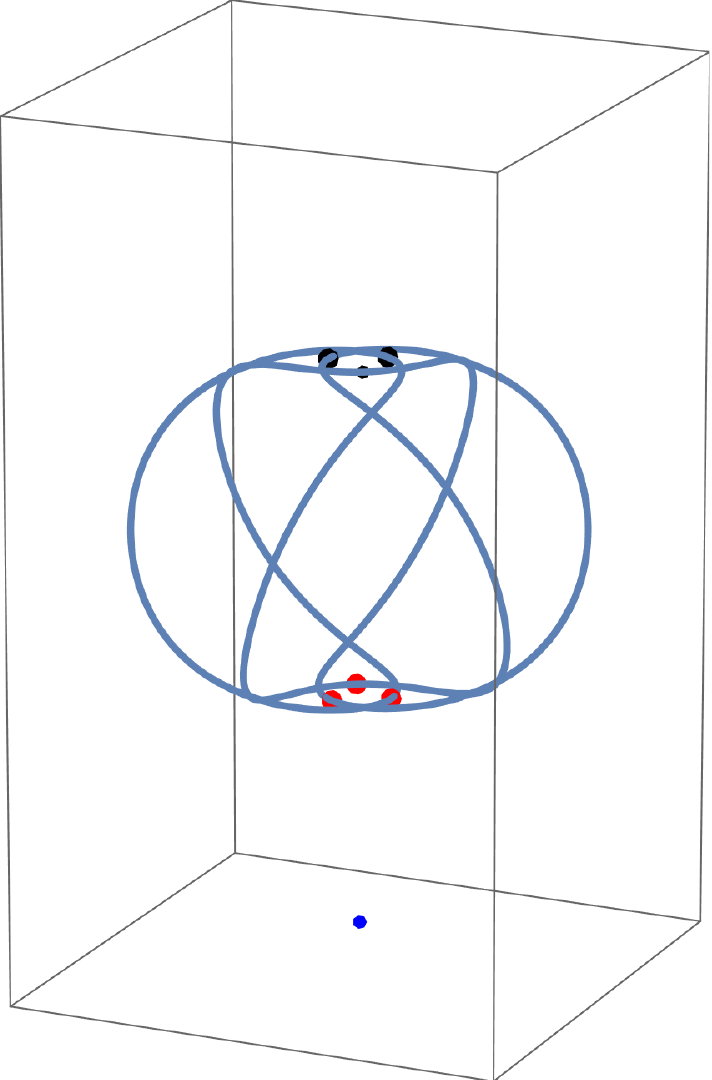}}
	\caption{Image of the (6+1)-bodies for the solution given by the parameters $a_1=0.581722$, $b_1=0.81081$, $u_1=1.96752$ and $T_1=6.53474$. From left to right we show the bodies when $t=T_1/2$, $t=3T_1/4$, $t=T_1$, $t=5 T_1/4$ and $t=3T_1/2$. }
	\label{bex1}
\end{figure}


\subsection{Example 2.}  We can check that the choice $a_2=1.37168$, $b_2= 0.717282$, $u_2=1.73494$ and $T_2=6.95831$
is a numerical solution of the system. 
\begin{eqnarray}\label{system}
 \dot{R}(a,b,T)=D(a,b,T)=Z(a,b,u,T)=0.
\end{eqnarray}
For these values of $a$ and $b$, the primaries share three trajectories. As expected, the functions $d(t)=D(a_2,b_2,u_2,t)$, $r(t)=R(a_2,b_2,u_2,t)$ and $Z(t)=Z(a_2,b_2,u_2,t)$ are functions with period $2T_2$. Figure \ref{fex2} shows the functions $d(t)$, $r(t)$ and $z(t)$. Figure \ref{bex2} shows the (6+1)-bodies for different values of $t$.
\begin{figure}[h]
	\centerline
	{\includegraphics[scale=0.87]{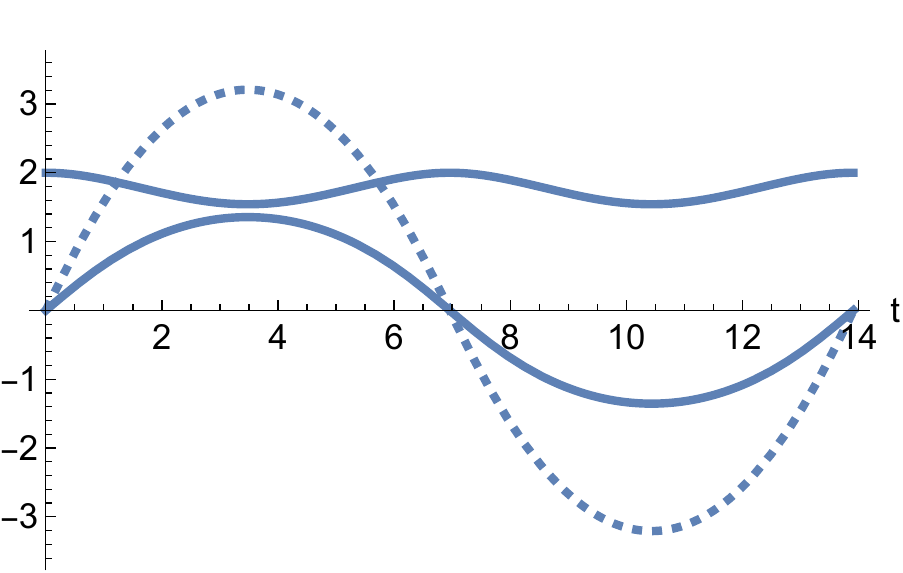}}
	\caption{Graph of the $2T_2$-periodic functions $d(t)$, $r(t)$ and $z(t)$ given by the parameters  $a_2=1.37168$, $b_2= 0.717282$, $u_2=1.73494$ and $T_2=6.95831$. The dashed graph is the function $z(t)$ and the function that starts at $r_0=2$ is the function $r(t)$.  }
	\label{fex2}
\end{figure}

\begin{figure}[h]
	\centerline
	{\includegraphics[scale=0.41]{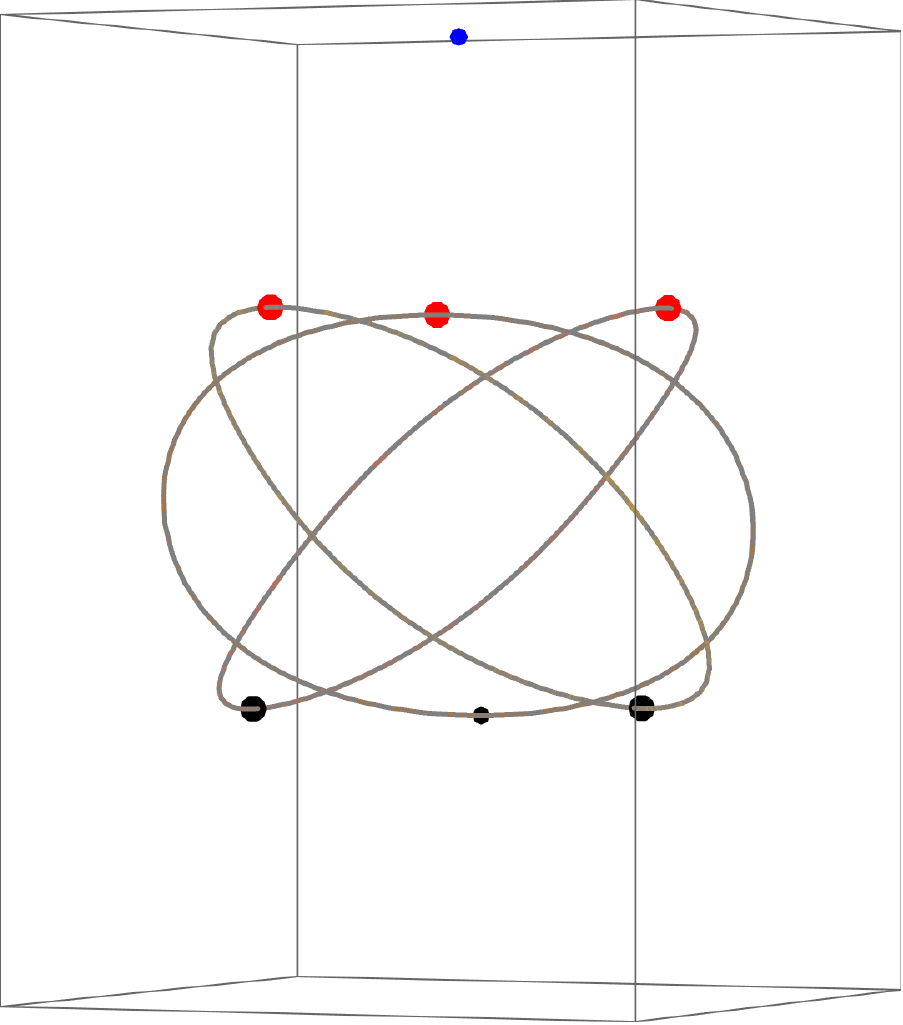} \includegraphics[scale=0.41]{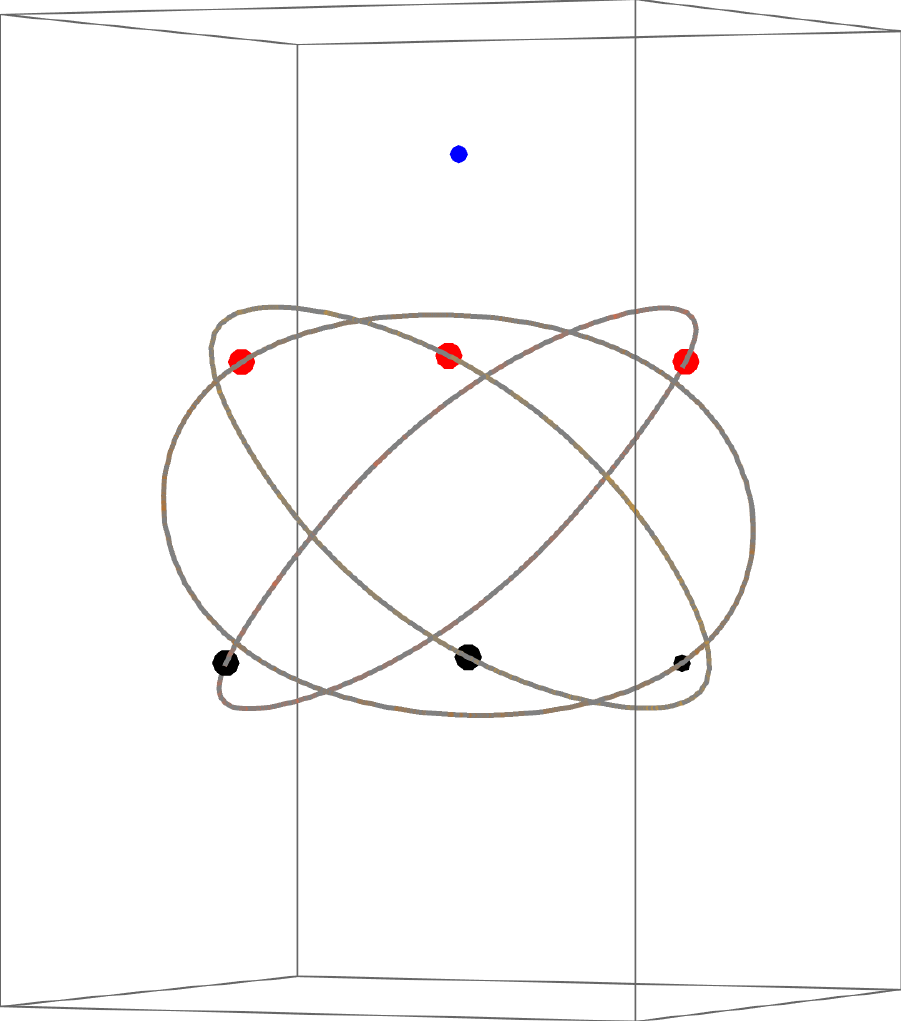}\includegraphics[scale=0.35]{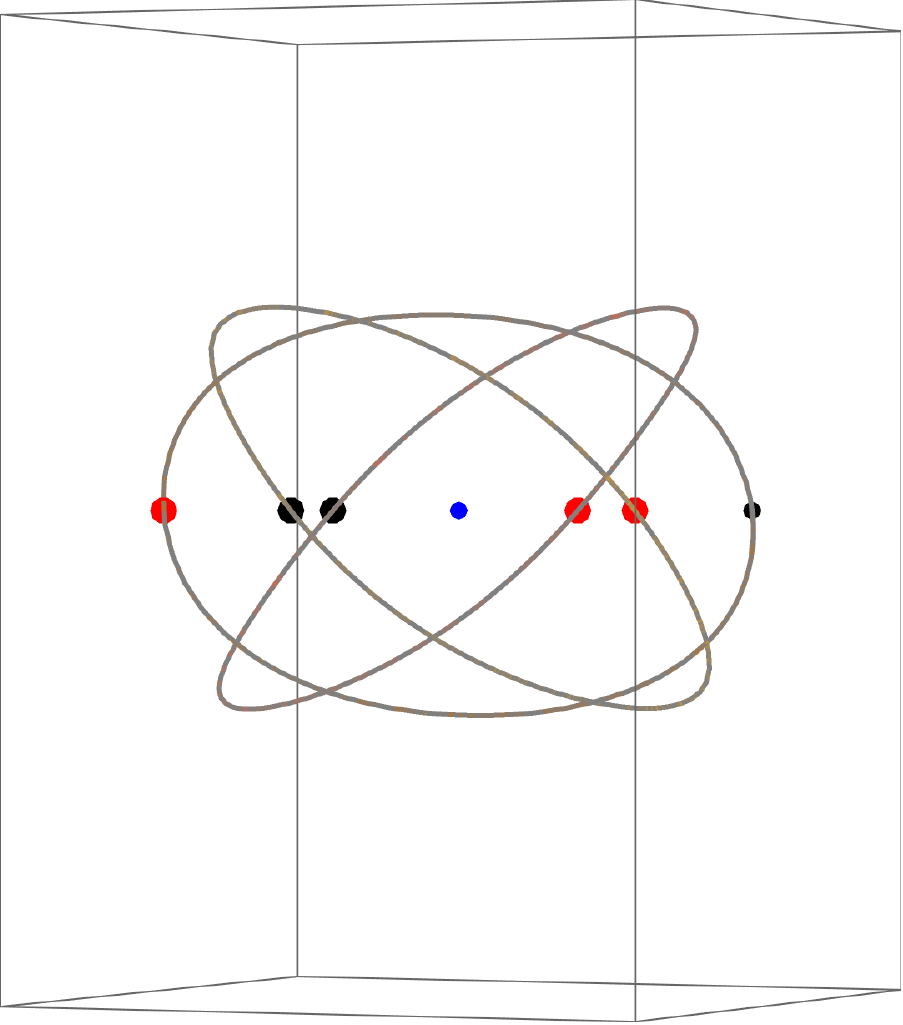}\includegraphics[scale=0.41]{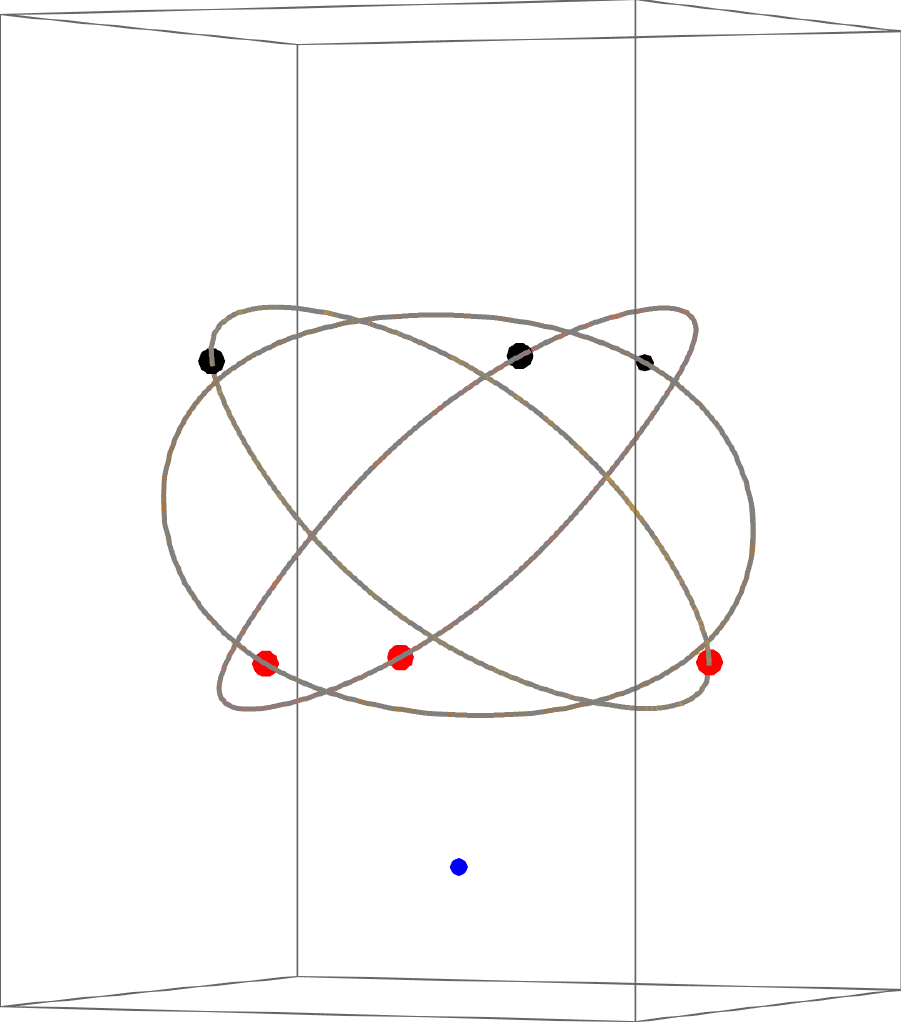}\includegraphics[scale=0.35]{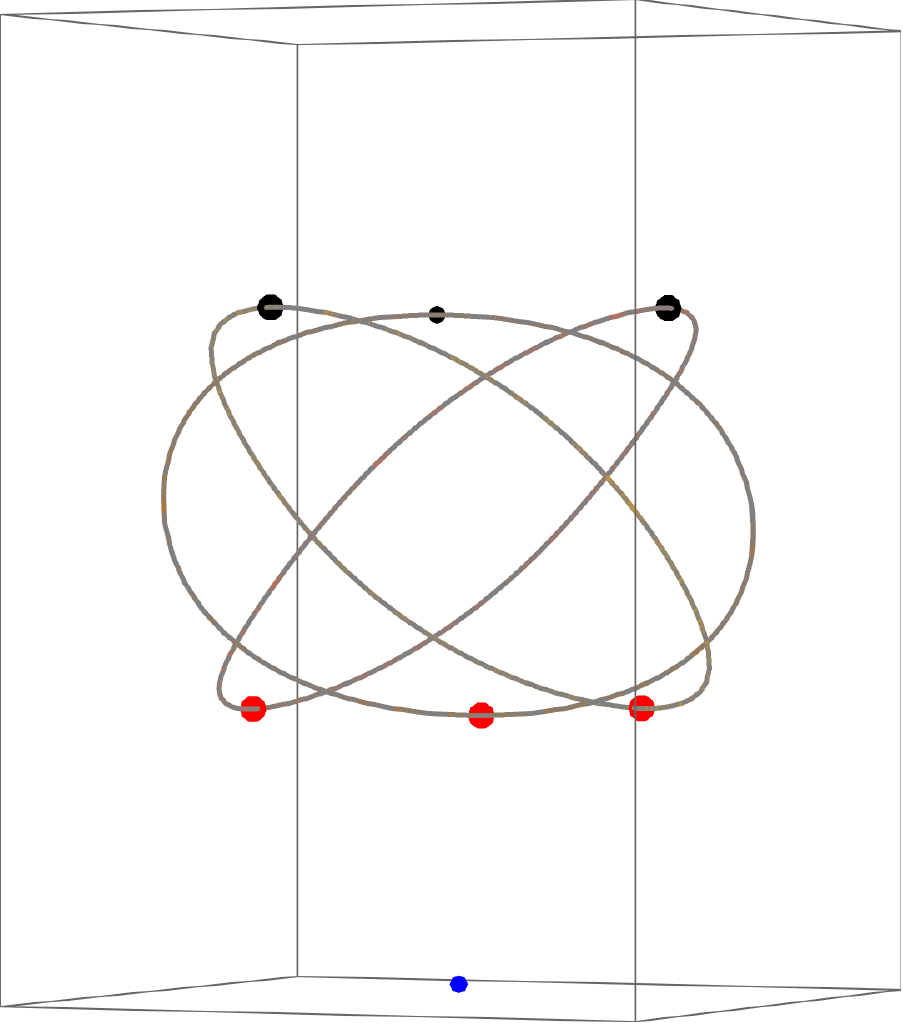}}
	\caption{Image of the (6+1)-bodies for the solution given by the parameters  $a_2=1.37168$, $b_2= 0.717282$, $u_2=1.73494$ and $T_2= 6.95831$. From left to right we show the images when $t=T_2/2$, $t=3T_2/4$, $t=T_2$, $t=5 T_2/4$ and $t=3T_2/2$ }
	\label{bex2}
\end{figure}

\vfil
\eject

\end{document}